\documentclass[11pt,a4paper,twoside]{amsart}

\usepackage{amsmath}
\usepackage{amssymb}
\usepackage{amsthm}
\usepackage{graphicx}
\usepackage{nicefrac}
\usepackage{tikz-cd}

\newtheorem{theorem}{Theorem}[section]
  
  \newtheorem{lemma}[theorem]{Lemma}
  \newtheorem*{thmsys-lam}{Theorem~\ref{thm:sys-lam}}
  \newtheorem*{thmlam-sys}{Theorem~\ref{thm:lam-sys}}
  \newtheorem*{thmequi-iso}{Theorem~\ref{thm:equi-iso}}

\theoremstyle{definition}
  \newtheorem{definition}[theorem]{Definition}
  \newtheorem{example}[theorem]{Example}

\theoremstyle{remark}
  \newtheorem{remark}[theorem]{Remark}

\newcommand\B{\mathcal{B}}              
\newcommand\Z{\mathbb{Z}}               
\newcommand\N{\mathbb{N}}               
\newcommand\mR{\mathbb{R}}              
\newcommand\K{\mathcal{K}}              
\newcommand\De{\mathfrak{D}}            
\newcommand\Bas{\mathfrak{B}}           
\newcommand\D{\mathbb{D}}               
\newcommand\Sec{\mathcal{S}}            
\newcommand\st{\mid}                    
\newcommand\bigst{\bigm|}               
\DeclareMathOperator\dom{dom}           
\DeclareMathOperator\sing{Sing}         
\newcommand\A{\mathcal{A}}              
\newcommand\cL{\mathcal{L}}             
\DeclareMathOperator\id{\mathrm{id}}    
\DeclareMathOperator\cod{codim}         

\title{Codimension zero laminations are inverse limits}
\author{\'Alvaro Lozano Rojo}
\email{alvarolozano@unizar.es}
\address{Centro Universitario de la Defensa.
Academia General Militar.
Ctra. Huesca s/n.
E-50090 Zaragoza, Spain - and - IUMA, Universidad de Zaragoza.}



\begin{document}

\begin{abstract}
  The aim of the paper is to investigate the relation between inverse limit 
  of branched manifolds and codimension zero laminations. We give necessary 
  and sufficient conditions for such an inverse limit to be a lamination. We 
  also show that codimension zero laminations are inverse limits of branched 
  manifolds.
  
  The inverse limit structure allows us to show that equicontinuous 
  codimension zero laminations preserves a distance function on transversals.
\end{abstract}

\maketitle

\section{Introduction}
\label{sec:intro}

Consider the circle $\mathbb{S}^1=\{\,z\in\mathbb{C}\st\lvert z\rvert=1\,\}$ 
and the cover of degree $2$ of it $p_2(z)=z^2$. Define the inverse limit
\[
  \mathbf{S}_2 = \varprojlim(\mathbb{S}^1,p_2)
               = \bigl\{\, (z_k) \in{\textstyle\prod_{k\geq0}\mathbb{S}^1}
                                                \bigst z_k^2=z_{k-1} \,\bigr\}.
\]
This space has a natural foliated structure given by the flow $\Phi_t(z_k) = 
( e^{2\pi i t/2^k}z_k )$. The set $X=\{\,(z_k)\in\mathbf{S}_2\st z_0 = 1\,\}$ 
is a complete transversal for the flow homeomorphic to the Cantor set. This 
space is called \emph{solenoid}. This construction can be generalized 
replacing $\mathbb{S}^1$ and $p_2$ by a sequence of compact $p$-manifolds and 
submersions between them. The spaces obtained this way are compact laminations 
with $0$ dimensional transversals.
\medskip


This construction appears naturally in the study of dynamical systems. 
In \cite{williams:ODNWS, williams:EA} R.F.~Williams proves that an expanding 
attractor of a diffeomorphism of a manifold is homeomorphic to the inverse 
limit 
\[
  S\xleftarrow{\;f\;} S \xleftarrow{\;f\;} S \xleftarrow{\;f\;} \cdots
\]
where $f$ is a surjective immersion of a branched manifold $S$ on itself. A 
branched manifold is, roughly speaking, a CW-complex with tangent space at 
each point.

After their introduction by Williams, branched manifolds and their limits have 
been extensively used in the study of dynamical systems and foliations. For 
example W.~Thurston uses \emph{train tracks} (1-dimen\-sional 
branched manifolds) in geodesic laminations on hyperbolic surfaces 
\cite{thurston-3mani}. Later, J.~Anderson and I.~Putnam show \cite{AP:TISTC*A} 
that substitution tiling spaces are inverse limits of a CW-complex as in the 
case of Williams. J.~Bellisard, R.~Benedetti and J-M.~Gambaudo 
\cite{BBG:STFTAGL}, F.~G\"ahler \cite{gahler:CQT} and L.~Sadun 
\cite{sadun2003} independently extended this result to any tiling, showing 
that they are inverse limits of branched manifolds. In this case, the 
projective system has different branched surfaces at each level. With a 
similar scheme as in \cite{BBG:STFTAGL} R.~Benedetti and J.M.~Gambaudo has 
extended in \cite{BG:DGSADS} the previous result to $\mathbb{G}$-solenoids 
(free actions of a Lie group $\mathbb{G}$ with transverse Cantor structure).

F.~Alcalde Cuesta, M.~Macho Stadler and the author prove in \cite{ALM:TCLIL} 
that any compact without holonomy minimal lamination of codimension $0$ is an 
inverse limit, generalizing all previous results.
\medskip

In this article, we thoroughly explore the relation between inverse limits of 
branched manifolds and laminations with zero dimensional transverse structure. 
Let us start considering the following example. Take the eight figure
\[
  K = \mathbb{S}^1\wedge_1\mathbb{S}^1 
    = \nicefrac{\mathbb{S}^1\times \{1,2\}}{(1,1)\sim(1,2)},
\]
that is, the two copies of the circle glued by the $1$. For each copy of 
$\mathbb{S}^1$ we have the degree two covering $p_2$, so we can define 
$P_2([z,i]) = [z^2,i]$, where $[z,i]\in K$. Let $X$ be the inverse limit 
$\varprojlim(K,P_2)$. It is easy to see that $X$ is homeomorphic to
$\mathbf{S}_2\wedge_{(1)} \mathbf{S}_2$, i.e., two copies of the solenoid 
glued by the sequence $(1)\in \mathbf{S}_2$. It is clear that it is not a 
lamination. The problem is that $P_2$ does not \emph{iron out} the branching 
in each step, collapsing the branches at branching points to one single disk. 
The kind of maps doing that are called \emph{flattening} \cite{BBG:STFTAGL}. 
\medskip

We have three main theorems in the paper. Firstly, we show that this is a 
necessary and sufficient condition on an inverse systems to obtain a 
lamination as it limit:

\begin{thmsys-lam}
  Fix a projective system $(B_k,f_k)$ where $B_k$ are branched $n$-manifolds 
  and $f_k$ cellular maps, both of class $C^r$. The inverse limit $B_\infty$ 
  of the system is a codimension zero lamination of dimension $n$ and class 
  $C^r$ if and only if the systems is flattening.
\end{thmsys-lam}

Secondly, thinking in laminations as tiling spaces \cite{ALM:TCLIL} we can 
adapt the constructions for tilings \cite{gahler:CQT,sadun2003} to obtain a 
result in the other direction: from laminations to systems of branched 
manifolds. This theorem extends \cite{ALM:TCLIL} to any lamination of 
codimension zero:

\begin{thmlam-sys}
  Any codimension zero lamination $(M,\cL)$ is homeomorphic to an 
  inverse limit $\varprojlim(S_k,f_k)$ of branched manifolds $S_k$ and 
  cellular maps $f_k: S_k\to S_{k-1}$.
\end{thmlam-sys}

Finally, the inverse limit structure can give information of the dynamics 
of the space. M.C.~McCord \cite{mccord:ILSWCM} and recently A.~Clark and 
S.~Hurder \cite{CH:HMM} study an important class of solenoidal spaces, those 
given by real manifolds and regular covering maps as bounding maps. With this 
structure we can conclude that:

\begin{thmequi-iso}
  An equicontinuous lamination of codimension zero preserves a transverse 
  metric.
\end{thmequi-iso}

\noindent\textbf{Acknowledgements.} We would like to thanks Centre de Recerca 
Ma\-te\-m\`a\-ti\-ca for support during the Foliations Program were part of 
this article was written. We also thanks to F.~Alcalde Cuesta and M.~Macho 
Stadler for their valuable comments.


\section{Branched manifolds}
\label{sec:branched}

Let $\D^n$ denote the closed $n$-dimensional unit disk. A \emph{sector} is 
the, eventually empty, interior of the intersection of a (finite) family of 
half-spaces through the origin. Fix a finite family of sectors $\Sec$, and a 
finite directed tree $T$ with a map $s:VT\to\Sec$ where $VT$ is the vertex set 
of $T$.
\medskip

Now define a \emph{local branched model} $U_T$ as the quotient set of 
$\D^n\times VT$ by the relation generated by
\[
  (x,v)\sim (x,v') \iff
  \text{the edge $v\to v'$ exists in $T$ and $x\in \D^n\smallsetminus s(v)$}.
\]
The quotient $D_v\subseteq U_T$ of each set $\D^n\times\{v\}$ is called a 
\emph{smooth disks}. There is a natural map $\Pi_T: U_T \to \D^n$ given by the 
quotient of the first coordinate projection $pr_1 : (x,v) \mapsto x$, which is 
a homeomorphism restricted to each smooth disk.
\begin{figure}
    \begin{tikzcd}[
        column sep=tiny, every label={text depth=1.8cm},
        baseline=(current bounding box.center)
    ]
      \begin{tikzpicture}
        \fill[gray,opacity=0.2, draw=gray,thick, draw opacity=0.5]
                          (.5,0) -- (0,0) -- (0,-.5) arc (-90:0:5mm) -- cycle;
        \fill (0,0) circle (1.5pt);
        \draw (0,0) circle (5mm);
      \end{tikzpicture}
        \ar{dr}{} &   & \ar{dl}{}
      \begin{tikzpicture}
        \fill[gray,opacity=0.2, draw=gray,thick, draw opacity=0.5]
                          (.5,0) -- (0,0) -- (0,-.5) arc (-90:0:5mm) -- cycle;
        \fill (0,0) circle (1.5pt);
        \draw (0,0) circle (5mm);
      \end{tikzpicture} \\ &
      \begin{tikzpicture}
        \fill (0,0) circle (1.5pt);
        \draw (0,0) circle (5mm);
      \end{tikzpicture}
    \end{tikzcd}
    \hspace{1cm}
    \begin{tikzpicture}[
        join=round, line cap=round,baseline=(current bounding box.center)
    ]
       \draw (0,-.5) -- (0,0) -- (1,0)
                      to[out=155, in=55] node[pos=.2] (A) {} (0,-.5) -- cycle;

      \draw (0,-.5) -- (0,0) -- (1,0) .. 
                      controls (1.1,-.5) and +(.1,-.3) .. (0,-.5) -- cycle;
      \shade[bottom color=gray!70, top color=gray!10]
          (A.center) to[out=185,in=55] (0,-.5) -- (0,0) -- cycle;
      \draw (1,0) to[out=155, in=55] (0,-.5) -- (0,0) -- (A.center);
      \draw (0,0) ellipse (1cm and 5mm);
      \fill (0,0) ellipse (2pt and 1pt);
    \end{tikzpicture}

  \caption{A tree and the local model.}
  \label{fig:local_model}
\end{figure}
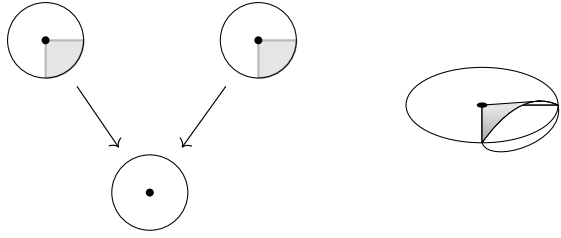

\begin{definition}
  \label{def:branched}
  A \emph{branched manifold of class $C^r$ of dimension $n$} is a Polish space 
  $S$ endowed with an atlas of closed disks $\{U_i\}$ homeomorphic to local 
  branched models of dimension $n$ such that there is a cocycle of 
  $C^r$-diffeomorphisms $\{\alpha_{ij}\}$ between open sets of $\D^n$ 
  fulfilling $\Pi_i\circ\alpha_{ij} = \Pi_j$, where $\Pi_i$ denotes the 
  natural map of the local models.
\end{definition}

\begin{remark}
  Definition~\ref{def:branched} is not the classical one \cite{williams:EA}. 
  In our setting the branching behavior is quite simple as we have locally
  finite branching. This is not true with the classical definition.
\end{remark}

Following \cite{williams:EA}, there is a natural notion of differentiable map: 
a map $f:S\to S'$ between two branched manifolds of class at least $r\geq 1$ 
is of class $C^r$ if the local map
\[
  f^{T'}_{T,v} : \Pi_T(D_v)
        \xrightarrow{\;(\Pi|_{D_v})^{-1}\;}
            D_v \xrightarrow{\;f\;} U_{T'} \xrightarrow{\;\Pi_{T'}\;} \D^n
\]
is of class $C^r$ and the germ of $f^{T'}_{T,v}$ at any point does not depends 
on $T'$.

Of course, there is the notion of tangent fibre bundle for branched manifolds. 
For each smooth disk $D_v$ of a local model $U_T$ we have the induced tangent 
bundle $(\Pi_T|_{D_v})^* T\mR^n \cong D_v\times\mR^n$. The tangent bundle of 
$U_T$ is just the \emph{gluing} of those, which is naturally isomorphic to 
$U_T\times\mR^n$. Finally, as the change of coordinates are diffeomorphisms 
between smooth disks there is a well defined tangent bundle. Obviously, one 
can define higher bundles provided the differentiability class is high enough. 
As in the classical setting differentiable maps induce bundle maps. Then the 
usual definitions of \emph{submersion} and \emph{immersion} make sense. 
Anyway, we should remark that a submersion does not need to be locally 
surjective and an immersion does not need to be locally injective.


\section{Codimension zero laminations}
\label{sec:mathbox}

Let $M$ be a locally compact Polish space. A \emph{foliated atlas of class 
$C^r$} is an atlas $\A = \{\varphi_i:U_i\to\mR^n\times X_i\}$ of $M$ such that 
the changes of coordinates can be written as
\begin{equation}
  \label{ec:cambio_coor}
  \varphi_{i}\circ\varphi_j^{-1}(x,y) = 
      \bigl(\varphi_{ij}^{y}(x),\sigma_{ij}(y)\bigr),
\end{equation}
where $\sigma_{ij}$ is a homeomorphism and $\varphi_{ij}^{y}$ is a 
$C^r$-diffeomorphism depending continuously on $y$ in the $C^r$-topology. Two 
foliated atlases of class $C^r$ are \emph{equivalent} if the union of them is 
still a foliated atlases. A \emph{lamination} of class $C^r$ is a base space 
$M$ endowed with an equivalence class $\cL$ of foliated atlas. The open sets 
$U_i$ are usually called \emph{flow boxes} and the level sets 
$\varphi_i^{-1}(P_i\times\{y\})$ and $\varphi_i^{-1}(\{x\}\times X_i)$ are the 
\emph{plaques} and \emph{local transversals} of $\A$ respectively. A 
\emph{transversal} of $\cL$ is just a union of local transversals of a 
compatible atlas. The \emph{leaves} $L$ of $\cL$ are the smallest path 
connected sets such that if $L$ meets a plaque $P$, then $P\subset L$. The 
plaques define a $n$-manifold structure of class $C^r$ on each leaf. It is 
usual to identify the lamination $\cL$ with the partition formed by the 
leaves, so it is common to write $L\in\cL$ for a leaf $L$ of $\cL$. Finally, a 
\emph{complete transversal} is a transversal meeting all leaves. The space 
$T=\bigsqcup_i X_i$ is the \emph{complete transversal associated to the atlas 
$\A$}. An atlas $\A$ is \emph{good} if it is finite and
\begin{enumerate}

  \item the flow boxes $U_i$ are relatively compact; 

  \item if $U_i \cap U_j\neq \emptyset$, there is a flow box $U_{ij}$ 
    containing $\overline{U_i \cap U_j}$ and therefore each plaque of $U_i$ 
    intersects at most one plaque of $U_j$.

\end{enumerate}
The codimension of the lamination $\cL$, denoted by $\cod\cL$, is the dimension 
of the complete transversal associated to any atlas. Then we talk about 
codimension zero lamination if the transversal is a locally compact Polish 
$0$-dimensional space.

\begin{remark}
  Given a codimension zero laminations, the leaves of the underlaying 
  laminations are the path connected components, as local transversals are 
  totally disconnected. Hence the foliated structure is intrinsic to the 
  topological space $M$.
\end{remark}

\begin{example}
  Consider the Cantor set $C=\{0,1\}^\Z$ and the shift map 
  $\sigma((x_i))_k = x_{k+1}$, for $(x_i)\in C$. We can consider then the 
  suspension of this homeomorphism: Consider the action of $\Z$ on 
  $\mR\times C$ given by
  \[
    (\lambda,x)+\ell = \bigl(\lambda-\ell,\sigma^\ell(x)\bigr),
  \]
  for $\ell\in\Z$. The horizontal foliation on $\mR\times C$ given by the 
  leaves $\mR\times\{\ast\}$ is invariant under this action, so it induces a 
  lamination of class $C^\infty$ on the quotient $M = \mR\times C/\Z$. For 
  other examples of codimension zero laminations see 
  \cite{ghys:LSR, blanc:EMMFS, lozano:EFDG, lozano:ENUEL}. 
\end{example}

Recall that a pseudogroup of transformations on a space $T$ is a family of 
homeomorphisms from open sets to open sets of $T$ closed under restriction to 
open sets, composition, inversion and extension. To describe the behavior of 
the leaves we need the \emph{holonomy pseudogroup}: Fix a good foliated atlas 
$\A$ for a lamination $\cL$ on $M$. Fix two meeting flow boxes $U_i$ and $U_j$ 
and define the corresponding \emph{holonomy transformation} 
$\sigma_{ij} : D_{ij} \subset X_i \to D_{ji} \subset X_j$ given by
\[
  \sigma_{ij}(y_i) = y_j \iff \text{the plaques through $y_i$ and $y_j$ meet,}
\]
where $D_{ij}$ and $D_{ji}$ are the obvious open sets of $X_i$ and $X_j$ 
respectively. It can be shown $\sigma_{ij}$ is a homeomorphism. The 
pseudogroup $\Gamma$ on $T$ generated by all those maps is the 
\emph{transverse holonomy pseudogroup associated to $\A$}.

Given a point $x\in T$ we can consider the group $\Gamma_x$ of germs of 
elements of $\Gamma$ fixing $x$. Given other element $y\in T$ in the same leaf 
$L$, $\Gamma_x$ and $\Gamma_y$ are naturally isomorphic. We call this group 
the \emph{holonomy group} of $L$ and we write $\Gamma_L$. In fact, the 
holonomy group is a quotient of the fundamental group of $L$: Take a path 
$\gamma$ starting and ending at $x$ contained in the corresponding leaf. Then 
it can be covered by a finite family charts. Making the composition of 
the corresponding holonomy maps, we obtain a map $\sigma_\gamma$ depending on 
the homotopy class of the path. We have the \emph{holonomy representation} 
of the fundamental group of $L$. Details about laminations can be found in 
\cite{CC:F1}.


\section{From sequences of branched manifolds to laminations}
\label{sec:lim-match}

For us, a \emph{projective system} is family of spaces $S_k$ and continuous onto maps
$f_k:S_{k+1}\to S_k$ with $k\in\N$. Given such a system it is possible to 
define the \emph{projective} or \emph{inverse limit} as
\[
  S_\infty = \varprojlim(S_k,f_k)
            = \bigl\{\, (b_k) \in{\textstyle\prod_{k\geq0} S_k}
                                          \bigst f_k(b_{k+1}) = b_k \,\bigr\}.
\]
For each $S_k$, there is a natural map $\hat{f}_k:S_\infty\to S_k$, the 
restriction of the corresponding coordinate projection. Projective limits has 
a universal property: if there are maps $g_k:Y\to S_k$ such that 
$f_k\circ g_k = g_{k-1}$ for a space $Y$ and all $k$, then there exists a 
unique continuous map $g_\infty:Y\to S_\infty$ such that $\hat{f}_k\circ 
g_\infty=g_k$.

In general inverse limits might be empty, but if each factor $S_k$ is compact 
and nonempty, the projective limit will be compact and non empty. Moreover, if 
the factors are Hausdorff the inverse limit will be too \cite{dugundji:T}.
\medskip

Given an infinite set $\{a_i\}_{i\in\N}\subset\N$ we define the 
\emph{telescoping} (associated to the family $\{a_i\}$) of the projective 
system $\{S_k,f_k:S_k\to S_{k-1}\}$ as the projective system
\[
  S_{a_0}
    \xleftarrow{\;f_{a_1}\circ\cdots\circ f_{a_0+2}\circ f_{a_0+1}\;} S_{a_1}
      \xleftarrow{\;f_{a_2}\circ\cdots\circ f_{a_1+2}\circ f_{a_1+1}\;} S_{a_2}
        \cdots
\]
The telescoping is, roughly speaking, a less detailed projective system, as we
have dropped some levels, but with this process no information is lost. It is 
straightforward to show:

\begin{lemma}[Thm.~2.7 of Apx.~2 \cite{dugundji:T}]
  A projective system and any telescopic contraction of it have the same 
  inverse limit.
\end{lemma}

\begin{definition}
  A cellular $f:S\to S'$ map between branched manifolds is \emph{flattening} 
  if for each $b\in S$ there exists a (normal) neighborhood $U$ such that 
  $f(U)$ is a smooth disk of $S'$. A projective system if \emph{flattening} if 
  there is a telescopic contraction of it with all bond maps flattening.
\end{definition}

\begin{theorem}
  \label{thm:sys-lam}
  Fix a projective system $(S_k,f_k)$ where $S_k$ are branched $n$-manifolds 
  and $f_k$ cellular maps, both of class $C^r$. The inverse limit $S_\infty$ 
  of the system is a codimension zero lamination of dimension $p$ and class 
  $C^r$ if and only if the system is flattening.
\end{theorem}

\begin{proof}
  Let us start supposing that the system is flattening, we can assume the maps 
  $f_k$ are flattening immersions. And denote by $f_{k,k'}$ the maps
  \[
    f_{k,k'} = f_k \circ f_{k-1}\circ \cdots \circ f_{k'+1} : S_k\to S_{k'},
  \]
  where $k > k'$. We will construct a foliated atlas for 
  $S_\infty = \varprojlim (S_k,f_k)$. 
  \medskip
  
  Given a thread $(x_k)\in S_\infty$ there are two possibilities:
    a) $x_{k_0}\in S_{k_0}\smallsetminus\sing S_{k_0}$ for some $k_0\in\N$, or
    b) $x_k\in\sing S_k$ for all $k$.
    \medskip
  
  \noindent{}a) As $x_{k_0}\notin\sing S_{k_0}$ there is a normal 
    neighbourhood $D_{k_0}$ of $b_{k_0}$ diffeomorphic to an Euclidean closed 
    disk. As maps $f_k$ are cellular the space $f_{k,k_0}^{-1}(D_{k_0})$ is a 
    finite disjoint union of copies of $D_{k_0}$ (indeed one for each element 
    of $f_{k,k_0}^{-1}(x_{k_0})$), each of them being again a normal 
    neighbourhood, for each $k > k_0$. Define the set
    \begin{equation}
      \label{ec:U_D_0}
      U_{D_{k_0}} =
        \bigl\{\,(x'_k)\in S_\infty
            \bigst x'_k\in f_{k,k_0}^{-1}(D_{k_0})\,\bigr\},
    \end{equation}
    which can be also obtained as the inverse limit of the systems 
    $(f_{k,k_0}^{-1}(D_{k_0}),f_k)$. Each $f_{k,k_0}^{-1}(D_{k_0})$ is
    canonically homeomorphic to $D_{k_0}\times f_{k,k_0}^{-1}(x_{k_0})$. But 
    $D_{k_0}\times f_{k,k_0}^{-1}(x_{k_0})$ with maps $\id_{D_{k_0}}\times 
    f_k$ is an inverse system, and the canonical homeomorphism of each level 
    pass to the limit, therefore $U_{x_{k_0}}$ is homeomorphic to 
    $\varprojlim(D_{k_0}\times f_{k,k_0}^{-1}(x_{k_0}), \id_{D_{k_0}}\times 
    f_k) = D_{k_0}\times \hat{f}^{-1}_{k_0}(x_{k_0})$. Finally, notice that
    $\hat{f}^{-1}_{k_0}(x_{k_0})$ is a closed and totally disconnected set. 
    \medskip
    
  \noindent{}b) In this case, the construction is similar, but now normal 
    neighbourhoods are no longer disks. Given a normal neighbourhood $D_0$ of 
    $x_0$ we define the closed set $U_{D_0}$ as in \eqref{ec:U_D_0}. We should 
    use flattening condition to get a product structure. Consider local 
    coordinates for $D_0$, that is, a finite disjoint union $X_0$ of disks, 
    and a map $\varphi_0:X_0\to D_0$. We can assume that each of them is a 
    smooth section for $f_0$. Now, for each point of $f^{-1}_0(x_0)$ we can 
    give local coordinates being smooth sections of $f_1$. All those local 
    sections form $X_1$ a disjoint union of disks, and we also have the 
    coordinate map $\varphi_1:X_1\to f^{-1}_0(D_0)$. There is an obvious map 
    $g_0$ from $X_1$ to $X_0$ defined by $f_0$. Continuing in this way we get 
    a projective system $(X_k, g_k)$ and a system of maps $\{\varphi_k\}$ to 
    the system $(f^{-1}_{k,0}(D_0),f_k)$. It is straightforward to show that 
    $\varphi_{k-1}\circ g_k = f_k\circ\varphi_k$. Therefore, 
    there is a continuous map 
    $\varphi_\infty: \varprojlim X_k\to \varprojlim f^{-1}_{k,0}(D_0)$. It is 
    obvious $\varphi_\infty$ is onto, so it is enough that see that it is 
    injective.
    
    Fix $(y_k)$ and $(y_k')\in\varprojlim X_k$ such that 
    $\varphi_\infty(y_k)  = \varphi_\infty(y'_k)$. Then, for any $k$
    $\varphi_k(y_k) = \varphi_k(y'_k)$. But $f_k$ is flattening, therefore 
    $y_{k-1}$ and $y'_{k-1}$ belongs to the same disk in $X_k$. Thus 
    $y_{k-1} = y'_{k-1}$. By induction we get $(y_k)=(y'_k)$.
    \bigskip
    
  The change of coordinates of all those maps are like in \eqref{ec:cambio_coor} 
  as the transversals are totally disconnected.
  \medskip

  Now, suppose flattening condition is not fulfilled, even after any 
  telescopic contraction. So there exists a thread $(x_k)\in S_\infty$ such 
  that for each point $x_k\in S_k$ there are two smooth sections $D_k^0$ and 
  $D_k^1$ of $f_k$ such that $f_k:D_{k+1}^i\to D_k^i$ is a diffeomorphism for 
  $i=0,1$. In particular, $f_k$ restricted to the union $D_{k+1}^0\cup 
  D_{k+1}^1$ is a diffeomorphism of branched manifolds. Therefore the inverse 
  limit is diffeomorphic to $D_{k+1}^0\cup D_{k+1}^1$ and it is contained in 
  $S_\infty$. Hence it cannot be a lamination.
\end{proof}

\section{From laminations to sequences of branched manifolds}
\label{sec:match-lim}

\subsection{Box decompositions}
\label{subsec:box_dec}

\begin{definition}
  \label{def:box_desco}
  A family of compact flow boxes
  $\B = \{\varphi_i : B_i \to  D_i \times C_i\}_{i=1}^m$
  is said to be a \emph{box decomposition of $(M,\cL)$} if
  \begin{enumerate}

    \item the family $\B$ covers $M$;

    \item each transversal $C_i$ is a clopen set of $X$;

    \item if $i\neq j$, the intersection of $B_i$ and $B_j$ agrees
      with the intersection of the vertical boundaries
      $\partial_\pitchfork B_i = \varphi_i(\partial D_i \times C_i)$ and
      $\partial_\pitchfork B_j = \varphi_j(\partial D_j \times C_j)$;

    \item each plaque of $B_i$ meets at most one plaque of $B_j$;

    \item the changes of coordinates are given by: 
      \begin{equation}
        \label{ec:cambio_coor_cajas}
        \varphi_i\circ\varphi_j^{-1}(x,y)
              = \bigl(\varphi_{ij}(x),\sigma_{ij}(y)\bigr).
      \end{equation}

  \end{enumerate}
  If the plaques of $\B$ are $p$-simplexes, each pair of plaques meets in a 
  common face. In this case, we may suppose that the maps $\varphi_{ij}$ are 
  linear and we will say that $\B$ is a \emph{simplicial box decomposition} 
  (or simply \emph{simplicial decomposition}) of $(M,\cL)$.
\end{definition}

\begin{theorem}[\cite{ALM:TCLIL}]
  \label{thm:exists_sbd}
  Any codimension zero compact $C^1$ lamination $(M,\cL)$ admits a
  simplicial box decomposition $\B$.
\end{theorem}

The proof of the Theorem is long, but the main points of it are the same of the
Classical Triangulation Theorem with the approach inaugurated by
A.~Weil \cite{weil:STR}.


\subsection{Patterns and boxes}

For now on, we fix a simplicial box decomposition 
$\K=\{\varphi_\Delta : B_\Delta \to  \Delta \times X_\Delta\}$, where $\Delta$ 
runs over a finite set of $n$-simplexes $\mathcal{P}$. By analogy, we will 
call those elements (proto)tiles. Each transversal $X_\Delta$ can be 
identified withthe set of centers of the triangles of the corresponding box 
$B_\Delta$, or in other words, each simplex $\Delta$ is pointed at its center. 
We denote by $X$ the complete transversal associated to $\mathcal{K}$, that 
is, the disjoint union $\bigsqcup_\Delta X_\Delta$.

\begin{definition}
  A \emph{pattern} $P$ is a finite union (face by face) of copies of elements 
  of $\mathcal{P}$. If we fix a base point $p$ (chosen from the ones of the 
  tiles forming $P$) we talk about a \emph{pointed pattern $(P,p)$}. The base 
  point is dropped from notation if it is clearly understood.
\end{definition}

In general, patterns are not subsets of leaves of $\cL$, not even subsets of 
$M$. The natural place for patterns is the \emph{holonomy coverings of 
leaves}, that is, the covering associated to the kernel of the holonomy 
representation of the leaf.

\begin{definition}
  Let $\pi: \widehat{L}_x \to L_x$ denotes the holonomy covering of the leaf 
  $L_x$ through $x\in X$. A pointed pattern \emph{$(P,p)$ is around $x\in X$} 
  if there is a copy of the pattern $P\subset\widehat{L}_x$ such that 
  $\pi(p)=x$. Call 
  \[
    X_{(P,p)} = \bigl\{\, x\in X \bigst \text{$(P,p)$ is around $x$} \,\bigr\}.
  \]
  to the set of all those  points.
\end{definition}

\begin{lemma}
  The set $X_P$ is clopen.
\end{lemma}

\begin{proof}
  Fix $x\in X_P$. For each tile $\Delta$ 
  of $P$ we can fix a path of tiles $\{\Delta_0,\Delta_1,\ldots,\Delta_k\}$, 
  where $\Delta_0$ is the tile containing $p$ and $\Delta_k=\Delta$. The 
  projection of the path defines a path of plaques on $M$, and the associated 
  holonomy map is given by
  \[
    \gamma_{p,\Delta} = 
        \sigma_{\Delta_{k-1},\Delta_k}\circ
                \cdots\circ\sigma_{\Delta_1,\Delta_2}
                                \circ\sigma_{\Delta_0,\Delta_1}.
  \]
  The map above does not depend on the path itself but on $\Delta$ as the 
  different paths are homotopic relatively to its ends in the holonomy 
  covering. It is obvious that 
  \[
    X_P = \bigcap_{\Delta\in P} \dom\gamma_{p,\Delta}
  \]
  By definition of the holonomy transformations 
  $\sigma_{\Delta,\Delta'}$, their domains are clopen so 
  $\dom\gamma_{p,\Delta}$ are also clopen. Finally, there are a finite 
  number of simplexes, hence $X$ is clopen.
\end{proof}

Now, given a pattern $P$ there is a natural map $P\times X_P\to M$ given as 
the gluing of the embeddings $\Delta\times X_P\subset \Delta\times X_\Delta\to 
M$. Those maps shows the following result:

\begin{lemma}
  \label{lemma:singular-box}
  Given a pointed pattern $(P,p)$, there is a singular box
  \[
    \psi_P: P\times X_P \to M,
  \]
  that is, a foliated map such that
  \begin{enumerate}
    
    \item is a homeomorphism restricted to each transversal $\{\ast\}\times 
      X_P$ and

    \item is a local homeomorphism restricted to each plaque 
      $P\times\{\ast\}$.

  \end{enumerate}
\end{lemma}

\begin{definition}
  \label{def:decoration}
  A \emph{tile decoration} of $\K$ is a disjoint clopen cover of $X_\K$.
\end{definition}

It is clear the associated set $X_P$ for a decorated pattern is also clopen, 
as if is enough to restrict the domains and images of the holonomy 
transformations to the clopen sets of the decoration. Hence, 
Lemma~\ref{lemma:singular-box} also holds for decorated patterns.

\subsection{Codimension zero laminations are inverse limits}

The aim of this section is to prove the Theorem:

\begin{theorem}
  \label{thm:lam-sys}
  Any codimension zero lamination $(M,\cL)$ is homeomorphic to an inverse 
  limit $\varprojlim(S_k,f_k)$ of branched manifolds $S_k$ and submersions 
  $f_k: S_k\to S_{k-1}$.
\end{theorem}

\begin{proof}
  For now on let $\K$ denote a simplicial box decomposition as the one of
  Theorem~\ref{thm:exists_sbd}. The axis $X_\K$ of $\K$ is a totally 
  disconnected compact metrizable space, therefore we can choose a countable 
  basis $\Bas$ of the topology of $X_\K$ such that
  \begin{enumerate}
    
    \item each set $B\in\Bas$ is clopen and;
    
    \item the basis is written as a disjoint union of \emph{floors}
      $\Bas = \bigsqcup \Bas_k$ such that
      \begin{enumerate}
        
        \item for each $B\in \Bas_{k+1}$ there exists $B'\in\Bas_k$ such that 
          $B\subsetneq B'$ and;
          
        \item $\Bas_k$ is an open cover of $X_\K$.
        
      \end{enumerate}
  \end{enumerate}
  For each $p\in X_\K$, let $\widetilde{L}_p$ be the holonomy covering of 
  $L_p$ the leaf through $p$. Fix a lift $\tilde{p}$ of $p$. Let $\Delta_p$ 
  denote the simplex containing $p$, and $\Delta_{\tilde{p}}$ the lift of it 
  to $\widetilde{L}_p$. In other words, $\Delta_{\tilde{p}}$ is the simplex 
  whose barycenter is $\tilde{p}$.
  \bigskip
  
  \noindent\emph{The first branched manifold $S_1$.} Fix the decoration of 
  tiles $\De_1 = \Bas_1$. From now until next step all patters will be decorated 
  by $\De_1$. Consider the pattern $P^1_p=\star(\Delta_{\tilde{p}})$. It is 
  clear $P^1_p$ does not depend on the election of $\tilde{p}$, as deck 
  transformations relate all possible elections. Associated to $P^1_p$ there 
  is a clopen set $X_{P_p^1}$. By definition it is obvious that two sets 
  $X_{P_p^1}$ and $X_{P_{p'}^1}$ agree or are disjoint. As $X_\K$ is compact, 
  there are finitely many of them, or equivalently, there are finitely many 
  decorated patterns $\{P_1^1, P_2^1,\ldots,P_{\ell_1}^1\}$. We have then the 
  associated sets $X_{P_i^1}$ and singular boxes 
  $\varphi_i^1: P_i\times X_{P_i^1}\to M$. We define an equivalence relation 
  $\sim_1$ on $M$ generated by 
  \[
    x \sim_1 y
      \iff
        \text{$\exists\varphi_i^1$ s.t. 
          $pr_1\circ{\varphi_i^1}^{-1}(x)
                          = pr_1\circ{\varphi_i^1}^{-1}(y)$.
        }
  \]
  That relations is just the collapsing the transversals of the singular boxes 
  on $M$. Set $S_1 = M/\!\!\sim_1$ and $q_1$ the quotient map.
  \smallskip

  The space $S_1$ has a natural branched manifold structure. It is obvious 
  $S_1$ is covered by the images of the starts of points 
  $\star x = \bigcup_{\Delta\ni x} \Delta$, with $x\in M$. Moreover, by the 
  compactness of $S_1$ (there are just finitely many patterns on $M$) finitely 
  many of them are enough, and can be chosen to be vertices of the 
  triangulation. Sadly, those disks might not map to disks on $S_1$, as some 
  of their simplexes may meet in other ways in $M$, so the quotient is not a 
  disk (see Figure~\ref{fig:starDelta}a). Define the disk $D_x$ as
  \[
    D_x = \bigl\{\, y \in \star x 
                        \bigst \mathrm{bar}_x(y) \geq \nicefrac23 \,\bigr\},
  \]
  where $\mathrm{bar}_x(y)$ denotes the barycentric coordinate of $y$ 
  associated to $x$ on the corresponding simplex (see 
  Figure~\ref{fig:starDelta}a). Those disks are embedded as disks on $S_1$. 
  Now, given $s\in S_1$ the image of a vertex, define the open set
  \[
    U_s = \bigcup_{q_1(x\in M) = s} q_1(D_x).
  \]
  It is straightforward to see the $U_s$ open sets with the disks $D_x$ form a 
  branched manifold atlas for $S_1$.
  \bigskip

  \noindent\emph{The branched manifolds $S_k$.}
  Now, by induction suppose we have the partition by clopen sets 
  $\{X_{P_i^{k-1}}\}_{i=1}^{\ell_{k-1}}$ and $S_{k-1}$. We can define the 
  decoration
  \[
    \De_k = \bigl\{\,
                      X_{P_i^{k-1}} \cap B \bigst
                      \text{$1 \leq i\leq \ell_{k-1}$ and $B\in\B_k$} 
            \,\bigr\}.
  \]
  So now, for each $p\in P$, consider the pattern decorated by $\De_k$ 
  \[
    P^k_p=\star P^{k-1}_p = 
                  \star \stackrel{k}{\cdots}\star \Delta_{\tilde{p}}, 
  \]
  using the same notation 
  above. Proceeding as before, we obtain finite partition by clopen sets 
  $\{X_{P_i^k}\}_{i=1}^{\ell_k}$ and the branched manifold $S_n$ and the 
  quotient map $q_k$. 
  \bigskip 

  \noindent\emph{The branched manifold system.} By construction given $P_i^k$ 
  there exists $P_j^{k+1}$ such that $X_{P_j^{k+1}}\subset X_{P_i^k}$. It is 
  clear, $P_j^{k+1}$ is just the star of $P_i^k$ with a finer decoration. Even 
  more, the natural maps embedding the associated boxes agree. This implies 
  there is a natural map $f_{k+1}:S_{k+1}\to S_k$ such that 
  \[
    \begin{tikzcd}
      M           \ar{dr}{q_3}
                  \ar[bend left=10]{drr}{q_2}
                  \ar[bend left=15]{drrr}{q_1} \\
                  \ar[-,draw=none]{r}%
                        [near end, description]{\mbox{\normalsize$\cdots$}}
          & S_3   \ar{r}[below]{f_3}
          & S_2   \ar{r}[below]{f_2}
          & S_1
    \end{tikzcd}
  \]
  is commutative. Therefore, there is a continuous map 
  $q_\infty:M\to S_\infty$.
  \bigskip
  
  \noindent\emph{Finishing the proof. The map $q_\infty$ is a homeomorphism.}
  As each quotient map $q_n$ is onto, the map $q_\infty$ has dense image. 
  But $M$ is compact and $S_\infty$ is Hausdorff, therefore $q_\infty$ is 
  onto. On the other hand, $q_\infty$ is into. Fix $x$ and $y \in M$ such that 
  $q_\infty(x)=q_\infty(y)$. This implies those points share decorated 
  patterns of any size. Therefore $x = y$. As $q_\infty$ is continuous 
  bijection from a compactum to a Hausdorff space, it is a homeomorphism.
\end{proof}

\begin{figure}%
  \begin{tikzpicture}[
    scale=1.2,join=round, line cap=round,commutative diagrams/every diagram, 
    baseline=(current bounding box.south)
  ]

    \coordinate (a1) at (0  ,0);
    \coordinate (a2) at (1  ,0);
    \coordinate (a3) at (0.75,1);
    \coordinate (b1) at (-0.25,1.5);
    \coordinate (b2) at ( 1.5 ,1.5);
    \coordinate (c1) at (-0.65,0.75);

    \draw (a1) -- (a2) -- (a3) -- cycle;
    \draw (b2) -- (a2);
    \draw (b1) -- (a3) -- (b2) -- (b1) -- (a1) -- (a3) -- cycle;
    \draw (a1) -- (c1) -- (b1);

    \fill[gray,opacity=0.2, draw=gray,thick, draw opacity=0.5]
      (barycentric cs:a1=0,a2=2,a3=1)--(barycentric cs:a1=2,a2=0,a3=1) --
      (barycentric cs:a1=0,b1=2,a3=1)--(barycentric cs:b2=2,b1=0,a3=1) --
      cycle;

    \node (a3a) at (a3) {$\bullet$} node at (a3a.north) {$\scriptstyle x$};

    \node at (barycentric cs:b2=1,a2=1 ,a3=1) {$\scriptstyle\Delta$};
    \node at (barycentric cs:a1=1,c1=1 ,b1=1) {$\scriptstyle\Delta$};

    \node at (0.425,-0.5) {(a)};
    
  \end{tikzpicture}
  \hspace{2cm}
  \begin{tikzpicture}[
      scale=1.2,join=round, line cap=round,commutative diagrams/every diagram,
      baseline=(current bounding box.south)
  ]
    \coordinate (a1) at (0  ,0);
    \coordinate (a2) at (1  ,0);
    \coordinate (a3) at (0.75,1);
    \coordinate (b1) at (-0.25,1.5);
    \coordinate (b2) at ( 1.5 ,1.5);
    \coordinate (b3) at ( 1.5,-0.5);
    \coordinate (c1) at (-0.5,-0.75);

    \draw[thick] (a1) -- (a2) -- (a3) -- cycle;

    \draw (b2) -- (a2) -- (b3) -- cycle;
    \draw (b1) -- (a3) -- (b2) -- (b1) -- (a1) -- (a3) -- cycle;
    \draw (b1) -- (c1) -- (b3) -- (a1) -- (c1) -- cycle;

    \node at (barycentric cs:a1=1,a2=1 ,a3=1) {$\Delta$};
    
    \node at (0.5,-1) {(b)};
  \end{tikzpicture}
    
  \caption{
    (a) The set $\star x$ and the disk $D_x$ (in grey).
    (b) The pattern $\star\Delta$: all tiles meeting $\Delta$.
  }
  \label{fig:starDelta}
\end{figure}%

\begin{figure}
\end{figure}


\section{Inverse limits of regular cover maps}
\label{sec:equicontinos}

Adding conditions to the inverse limit structure, we can get interesting 
results about the transverse dynamics of the corresponding laminations. One 
interesting case is the equicontinuous foliated spaces. 

Consider a projective system
\begin{equation}
  \label{ec:inv_sys}
  S_1 \xleftarrow{\;f_2\;} S_2 
        \xleftarrow{\;f_3\;} S_3
          \xleftarrow{\;f_4\;} \cdots
\end{equation}
where $S_k$ are compact finitely dimensional manifold and $f_k$ are regular 
covering maps. To avoid trivial situations, assume all bonding maps are not 
homeomorphisms, that is, the degree of each $f_k$ is greater than $1$. The 
inverse limit $S_\infty = \varprojlim(S_k,f_k)$ has a natural structure of 
matchbox manifold as it is obviously flattening, even on those trivial cases.

\begin{theorem}
  Each standard transversal of $S_\infty$ has a natural structure of profinite 
  group. Moreover, those structures are all isomorphic.
\end{theorem}

\begin{proof}
  Fix a thread $(x_k)\in S_\infty$. For now on $x_k$ will be the base point of 
  $S_k$ for all homotopy and cover related tools.
  
  Consider the transversal $T_1 = q_1^{-1}(x_1)$. Now if $k > k'$,
  \[
    f_{k,k'} = f_k \circ f_{k-1}\circ \cdots \circ f_{k'+1} : S_k\to S_{k'}
  \]
  is also a regular covering map. By definition $f_{k,k-1} = f_k$. Denote by 
  $\Delta_{k,k'}$ the deck transformation group of $f_{k,k'}$ acting on 
  $D_{k,k'} = f_{k,k'}^{-1}(x_{k'})$. As $f_{k,k'}$ is regular, $\Delta_{k,k'}$ 
  acts freely and transitively on $D_{k,k'}$. 
  
  We have then the identification 
  \begin{equation}
    \label{ec:id_Detla_D}
    \delta\in\Delta_{k,1} \mapsto \delta(x_k)\in D_{k,1}.
  \end{equation}
  Therefore the transversal $T_1$ can be thought as
  \[
    T_1 =     q_1^{-1}(x_1) 
        =     \varprojlim(D_{k,1}     ,f_{k,k-1}|_{D_{k,1}}) 
        \cong \varprojlim(\Delta_{k,1}, f_{k,k-1}|_{D_{k,1}})
  \]
  with the previous identification. Let us see $f_{k,k-1}$ is a homomorphism 
  of groups considered with domain $\Delta_{k,1}$ and range $\Delta_{k-1,1}$.
  \medskip
  
  As $f_{k,1}$ is a regular covering $\pi_1(S_k)\trianglelefteq\pi_1(S_1)$ and 
  $\Delta_{k,1} = \nicefrac{\pi_1(S_1)}{\pi_1(S_k)}$ (with the usual 
  identifications). As the same holds with $f_{k,k'}$ we can think that 
  \[
    \cdots \trianglelefteq
      \pi_1(S_3) \trianglelefteq
        \pi_1(S_2)\trianglelefteq
          \pi_1(S_1).
  \]
  By the isomorphism theorem we have the identification
  \[
    \Delta_{k,1} = 
      \frac{\Delta_{k+1,1}}{\Delta_{k+1,k}}.
  \]
  Remember that $f_{k,k-1}$ is the quotient map
  $S_k \to \nicefrac{S_k}{\Delta_{k,k-1}} = S_{k-1}$. This, with the 
  identification \eqref{ec:id_Detla_D} and the normality of the groups
  \begin{equation}
    \label{ec:demonio}
    \begin{aligned}
      f_{k,k-1}(g) &\overset{\eqref{ec:id_Detla_D}}{=}
        f_{k,k-1}(g(x_k)) 
          =  \bigl\{\,
                h(g(x_k))         \bigst h\in\Delta_{k,k-1}
             \,\bigr\} \\
          &= \bigl\{\,
                g(\tilde{h}(x_k)) \bigst \tilde{h}\in\Delta_{k,k-1}
             \,\bigr\} 
          =  g\circ \bar{h} \bigl(
                \{\, \tilde{h}(x_k) \st \tilde{h}\in\Delta_{k,k-1} \,\} 
             \bigr)
    \end{aligned}
  \end{equation}
  for any $\bar{h}\in\Delta_{k,k-1}$. Again as $f_{k,k-1}$ is a quotient map 
  \[
    f_{k,k-1}(x_k) = x_{k-1} = 
      \bigl\{\, \tilde{h}(x_k) \bigst \tilde{h}\in\Delta_{k,k-1} \,\bigr\},
  \]
  so we conclude from \eqref{ec:demonio}
  \[
    f_{k,k-1}(g) = g\Delta_{k,k-1}.
  \]
  Therefore the transversal $T_1$ has the structure of the profinite group
  \[
    \varprojlim(\Delta_{k,1},f_{k,k-1}).
  \]
  Notice that the identity of the group is identified with $(x_k)$. But the 
  election of another thread does not change the isomorphism class.

  Finally, it is well known that the isomorphism class does not change by 
  telescoping, so for any transversal $q_k^{-1}(\tilde{x}_k)$ for any 
  $\tilde{x}_k\in S_k$ we obtain the same profinite group structure.
\end{proof}

All leaves of $S_\infty$ share the fundamental group $\bigcap_k \pi_1(S_k)$  
(see~\cite{mccord:ILSWCM}). Therefore each leaf $L$ is a normal covering space 
of $S_1$ having 
\[
  \Delta_{\infty,1} = \frac{\pi_1(S_1)}{\bigcap_k\pi_1(S_k)}
\]
as deck transformation group. The quotient map 
$\Delta_{\infty,1}\to\Delta_{k,1}$ gives rise to a natural representation of 
$\Delta_{\infty,1}$ in $T_1$:
\[
  g\in\Delta_{\infty,1}
    \mapsto (g\Delta_{k,1}) \in \varprojlim(\Delta_{k,1},f_{k+1,k}).
\]
This representation is faithful as if $g\in\Delta_{\infty,1}$ such that 
$g\Delta_{k,1} = \Delta_{k,1}$ for all $k$ implies $g$ acts trivially in all 
levels $S_k$. Then $g$ acts trivially on $L$, so it should be the identity.

In that way, $\Delta_{\infty,1}$ is a subgroup of $T_1$. Hence we have 
the free action of $\Delta_{\infty,1}$ on $T_1$ by left translations. We have 
then a diagonal action given by
\[
  \Psi: 
    h\cdot\bigl(\lambda, (\tilde{x}_k)_k \bigr) 
                               \mapsto \bigl( h^{-1}\lambda, h(x_k)_k \bigr).
\]
It is straightforward to see that:

\begin{lemma}
  \label{lemma:isaction}
  The lamination $S_\infty$ is the suspension of $\Psi$. Therefore the 
  transverse dynamics of $S_\infty$ given by the action of $\Delta_{\infty,1}$ 
  on $T_1$.\qed
\end{lemma}

As $T_1$ is a group we can consider a left invariant distance, in fact, we can 
construct it explicitly. Consider the discrete distance 
$\delta_{\Delta_{k,1}}$ on each group $\Delta_{k,1}$ which is $1$ on different 
elements. The product metric restricted to $T_1$
\begin{equation}
  \label{ec:distance_cod0}
  d\bigl( (x_k), (y_k) \bigr) 
          = \sum_{k=2}^\infty \frac1{2^k} \delta_{\Delta_{k,1}}(x_k,y_k)
\end{equation}
is invariant by the action of $T_1$ on itself, so it is invariant by 
the action of the subgroup $\Delta_{\infty,1}$. Therefore any lamination given 
as an inverse limit of manifolds and regular coverings preserves a distance 
function. A.~Clark and S.~Hurder shows in \cite{CH:HMM} that those laminations 
are exactly the equicontinuous ones. 

\begin{theorem}
  \label{thm:equi-iso}
  A equicontinuous lamination of codimension zero preserves a transverse 
  metric.
\end{theorem}



\end{document}